\documentclass[12pt, twoside, leqno]{article}

\usepackage{amsmath,amsthm}
\usepackage{amssymb}

\usepackage{enumerate}

\usepackage[T1]{fontenc}

\newtheorem{thm}{Theorem}[section]
\newtheorem{cor}[thm]{Corollary}
\newtheorem{lem}[thm]{Lemma}

\newtheorem{prob}{Problem}

\theoremstyle{definition}
\newtheorem{defin}[thm]{Definition}

\numberwithin{equation}{section}

\newcommand{\0}{\emptyset}
\newcommand{\mc}{\mathcal}
\newcommand{\RR}{\mathbb{R}}

\newcommand{\IFF}{\Leftrightarrow}
\newcommand{\IMP}{\Rightarrow}
\newcommand{\foralmostall}{\forall^\infty}
\newcommand{\existsinfty}{\exists^\infty}
\newcommand{\ii}{\mc{I}}
\newcommand{\jj}{\mc{J}}
\newcommand{\mm}{\mc{M}}
\newcommand{\nn}{\mc{N}}
\newcommand{\ff}{\mc{F}}
\newcommand{\bb}{\mc{B}}
\newcommand{\ee}{\mc{E}}
\newcommand{\kk}{\mc{K}}

\newcommand{\cm}{\mathfrak{c}}

\newcommand{\se}{\subseteq}
\newcommand{\es}{\supseteq}

\newcommand{\baire}{\omega^{\omega}}

\begin{document}


\baselineskip=17pt


\title{Universal sets for ideals}

\author{Aleksander Cie{\'s}lak\\
E-mail: aleksander.cieslak@pwr.edu.pl\\
\&
\\
Marcin Michalski\\
E-mail: marcin.k.michalski@pwr.edu.pl\\
\\
Department of Computer Science\\ 
Faculty of Fundamental Problems of Technology\\
Wroc{\l}aw University of Science and Technology\\
Wybrze{\.z}e Wyspia{\'n}skiego 27\\
50-370 Wroc{\l}aw, Poland}
\date{}

\maketitle


\renewcommand{\thefootnote}{}



\renewcommand{\thefootnote}{\arabic{footnote}}
\setcounter{footnote}{0}


\begin{abstract}
In this paper we consider a notion of universal sets for ideals. We show that there exist universal sets of minimal Borel complexity for classic ideals like null subsets of $2^\omega$ and meager subsets of any Polish space, and demonstrate that the existence of such sets is helpful in establishing some facts about the real line in generic extensions. We also construct universal sets for $\ee$ - the $\sigma$-ideal generated by closed null subsets of $2^\omega$, and for some ideals connected with forcing notions: $\kk_\sigma$ subsets of $\baire$ and the Laver ideal. We also consider Fubini products of ideals and show that there are $\Sigma^0_3$ universal sets for $\nn\otimes\mm$ and $\mm\otimes\nn$.
\end{abstract}

\section{Introduction and preliminaries}
We use the standard set-theoretic notation based on \cite{Jech}. Let $X$ be a Polish space. Throughout the paper $\bb$ will denote the family of Borel subsets of $X$, $\mm$ the $\sigma$-ideal of sets of the first category, $\nn$ the $\sigma$-ideal of null subsets of $X$, $\ee$ the $\sigma$-ideal of sets generated by closed sets of measure zero, and $\kk_\sigma$ a $\sigma$-ideal generated by compact subsets of $X$. We may sometimes write $\bb(X)$, $\mm(X)$, etc. if we work in many different spaces and the context is not clear. Also, we consider the above mentioned ideals in spaces where it is reasonable to do so, i.e. in the cases of $\nn$ and $\ee$ we work in $\RR$ or $2^\omega$ and in the case of $\mc{K}_\sigma$ we work in $\baire$ as it is a natural example of a Polish space which is not $\sigma$-compact.
\\
Now let us state a definition of universal sets (following \cite{Sri}).
\begin{defin}
Let $\mc{F}$ be a family of subsets of $X$. We say that $U\se \omega^\omega\times X$ is a universal set for the family $\ff$ if $(\forall F) ((F\in\ff) \IFF (\exists x\in\baire)( U_x=F)).$
\end{defin}
In a similar fashion we define universal sets for ($\sigma$-)ideals. Let $\ii$ be a nontrivial ($X\notin\ii$ and $\ii$ contains all singletons) ideal or $\sigma$-ideal. Let us recall that a base of $\ii$ is a family $\mc{F}\se\ii$ which is cofinal in $\ii$, i.e. for each $A\in\ii$ there is $A'\in\ff$ such that $A\se A'$.
\begin{defin}
We call a set $U\se \baire\times X$ universal for $\mc{I}$ if it is universal for a base of $\ii$, i.e. a family of vertical sections $\{U_x: x\in \baire\}$ is a base of $\mc{I}$ and $U_x\in\mc{I}$ for all $x\in\omega^\omega$.
\end{defin}
A reason for a switch to a base of ideal, instead the whole ideal, is that usually there are too many sets in the ideal. That is also why we will consider mainly ideals that posses a Borel base, that is, a base consisting of Borel sets. The existence of such a base guarantees that an ideal has a cofinality not greater than $\cm$ and so the existence of a universal set for the ideal is possible (although it may be a wild subset of $\baire\times X$). A good example of an ideal for which there is no universal set is the Marczewski ideal $s_0$ (let us recall that $A\in s_0$, if for every perfect set $P$ there is a perfect set $Q\se P$ for which $Q\cap A=\0$), since its cofinality is greater than $\cm$ (see \cite{JMS}, also \cite{BKW}).

Universal sets were introduced in the beginning of XXth century with the dawn of descriptive set theory and have been studied by Lebesgue, Luzin and Souslin among others. Classic results regarding universal sets include: for each $\alpha<\omega_1$ there is a $\Sigma^0_{\alpha}$ universal set for the family of $\Sigma^0_{\alpha}$ sets, there exists a $\Sigma^1_1$ universal set for $\bb$, and also a $\Sigma^1_1$ universal set for the family of analytic sets. Recently universal sets were studied by A. W. Miller in \cite{Mi}, where the author considers so called uniquely universal sets. A set $U$ is uniquely universal if it is open subset of $X\times Y$ and for each open $W\se Y$ there is exactly one $x\in X$ with $U_x=W$. Two problems posed in that paper were solved by A. Krzeszowiec in \cite{Krz}. The notion of universal sets with respect to ideals appeared e.g. in \cite{PaRe}, where the authors constructed $G_\delta$ and $F_\sigma$ universal sets for the $\sigma$-ideals of null and meager subsets of the Cantor space respectively.

We are interested in finding universal sets of a possibly low complexity, by which we mean occurring early in the Borel or projective hierarchy. The existence of universal sets of a low complexity may be also useful in establishing some facts about the real line in generic extensions, but first let us give some context.
\begin{defin}
Let $\mc{I}$ be a definable ideal of subsets of a Polish space and let $M\se N$ both be a transitive models of ZFC. We say that $\mc{I}$ is absolute if for every Borel code $b\in M$ we have $M\models \#b\in \mc{I}$ if and only if $N\models \#b\in \mc{I}$.
\end{defin}
\begin{defin}
Let $\ii$ and $\jj$ be $\sigma$-ideals of subsets of spaces $X$ and $Y$ respectively. We define the Fubini product of these ideals in the following way
$$
A\in \ii\otimes\jj\IFF (\exists B\in\bb(X\times Y))(A\se B \land \{x: B_x\notin\jj\}\in\ii).
$$
\end{defin}
We say that a pair $(\ii, \jj)$ has the Fubini property if for every Borel set $B\se X\times Y$ we have
$$
\{x: B_x\notin\jj\}\in\ii\IMP \{y: B_y\notin\ii\}\in\jj.
$$
If $(\ii, \ii)$ has the Fubini property then we simply say that $\ii$ has it. Thanks to the notion of universal sets we may give a relatively simple proof of the following theorem (see \cite{Kun}, Theorem 3.20; also \cite{RaZe}, Lemma 3.1).
\begin{thm}
Let $\ii$ be an absolute $\sigma$-ideal satisfying the Fubini property, for which there exists a universal set $U\in\bb(\baire\times\RR)$. Let $M$ be a transitive and countable model of ZFC. Then in the $Borel / \ii$ forcing extension $M[G]$ we have $\RR\cap M\notin \mc{I}$.
\end{thm}
\begin{proof}
Let us assume that in a generic extension $M[r]$, where $r$ is a generic real, we have $\RR\cap M \in \ii$, and let $U$ be a universal set for $\ii$. For any $x\in\RR\cap M[r]$ there exists a Borel measurable $f\in M$ such that $f(x)=r$, so
$$
\RR\cap M \se \overline{U}_{r},
$$
where $\overline{U}=\{(x,y)\in \RR^{2}: (f(x),y)\in U\}$. By the Fubini property for $\ii$ we have that in the model $M$ there exists $y\in \RR\cap M $ such that $\overline{U}^{y}\in \ii$. Then, by the absoluteness of $\ii$, we have that $\overline{U}^{y}\in \ii$ in $M[r]$ and so $r\in \overline{U}^{y}$, which contradicts the genericity of $r$.
\end{proof}
Let us note that in \cite{ReZa} the authors found classes of ideals that have the Fubini Property beyond $\mm$ and $\nn$.
\\
Let us recall that an ideal $\ii\se P(Y)$ is called Borel-on-$\Sigma_1^1$, if for every Borel set $B\se X\times Y$ the set $\{x: B_x\in\ii\}$ is $\Sigma^1_1$. Similarly we define being $\mc{A}$-on-$\bb$ for any classes $\mc{A}, \bb$.

\begin{lem}\label{Borel-on-Sigma11}
Let $\ii$ be a Borel-on-$\Sigma_1^1$ ideal possessing a Borel base of bounded Borel complexity. Then there exists a universal set for $\ii$ of the same complexity as the base of $\ii$.
\end{lem}
\begin{proof}
Without loss of generality let us assume that the Borel base of $\ii$ has a complexity $\Sigma^0_\alpha$. Let $U$ be a universal $\Sigma^0_\alpha$ set and $A=\{x\in\baire: U_x\in\ii\}$. Since $A$ is analytic, there is a continuous and surjective function $f: \baire \rightarrow A$. Then the set
$$
V=(f, id)^{-1}[U], \textnormal{ where for all $(x,y)$ $(f, id)(x,y)=(f(x),y)$}
$$
is a universal set for $\ii$ of complexity $\Sigma^0_\alpha$.
\end{proof}
Let us observe that the existence of a universal set of bounded Borel complexity does not imply that the ideal is Borel-on-$\Sigma^1_1$: e.g. $\ii=\{\0\}$ is Borel-on-$\Pi^1_1$ and $\Pi^1_1$ cannot be substituted with $\Sigma^1_1$ in this case.
\\
The following lemma, which can be found in \cite{Bar} (Theorem 2.2.4), will be useful in the next section.
\begin{lem}\label{Meagery B-J}
Let $\mc{P}$ be a family of partitions of $\omega$ into intervals of the form $I_n=[a_n, b_n)$. A family $\{F_{x, P}\se 2^\omega: x\in 2^\omega, P\in\mc{P}\}$, where $F_{x, P}=\{y\in 2^\omega: (\foralmostall n\in\omega) (x|_{I_n}\neq y|_{I_n})\}$, is a Borel base of $\mm(2^\omega)$.
\end{lem}
If $x\in 2^\omega$ and $P$ is a partition of $\omega$ into intervals as in the lemma above, then
$$
F_{x, P}=\bigcup_{n\in\omega}\bigcap_{k>n}\bigcup_{m\in I_k}\{y\in 2^\omega: x(m)\neq y(m)\},
$$
so the complexity of the Borel base mentioned in the Lemma \ref{Meagery B-J} is $F_\sigma$.
\\
The main result of this paper includes also a construction of a universal set for the Laver ideal. Let us recall its definition.

\begin{defin}
We say that $A\se \omega^{\omega}$ is strongly dominating if for every $\Phi:\omega^{<\omega}\rightarrow \omega$ there exists $f\in A$ such that $\foralmostall_{n}f(n)\geq\Phi(f|_{n})$.
\end{defin}

The Laver ideal is the family of sets which are not strongly dominating (see \cite{Zap}).

\section{Acquiring universal sets}

In this section we will construct universal sets for the ideals mentioned above.

\begin{thm}
There exist:
\begin{description}
\item[(i)] a universal $F_\sigma$ set for countable subsets of $\baire$;
\item[(ii)] a universal $F_\sigma$ set for meager subsets of a Polish space $X$;
\item[(iii)] a universal $G_{\delta}$ set for null subsets of $2^{\omega}$ (equipped with the Haar measure);
\item[(iv)]a universal $F_{\sigma}$ set for $\mc{E}\se P(2^\omega)$;
\item[(v)]a universal $F_\sigma$ set for $\kk_{\sigma}\se P(\omega^{\omega})$;
\item[(vi)]a universal $G_{\delta}$ set for Laver ideal.
\end{description}

\end{thm}
\begin{proof}
Before we proceed with the constructions let us note that both $\sigma$-ideals of measure and category are Borel-on-Borel (see \cite{Kech}, Definition 18.5 and remarks), so by Lemma $\ref{Borel-on-Sigma11}$ both have universal sets of the same Borel complexity as their bases, which happens to be $G_\delta$ and $F_\sigma$ respectively. We will also give combinatorial proofs of these facts. The idea behind the construction in \textbf{(iii)} is based on \cite{PaRe}, we include the detailed proof for the sake of completeness and for the convenience of the reader. We also realize that \textbf{(i)} is a quite simple exercise, but we give the proof as a warm-up.
\vspace{12pt}
\\
\textbf{(i) Countable subsets of $\baire$:} Let $b$ be a bijection between $\omega\times\omega$ and $\omega$ and set a homeomorphism $h: \omega^\omega\rightarrow \omega^{\omega^\omega}$ given by the formula $(h(x)(m))(n)=x(b(m,n))$ for all $x\in \omega^\omega$. Then define $C\se \omega^\omega\times \omega^\omega$ by
$$
(x,y)\in C \IFF (\exists n\in\omega)(h(y)(n)=x).
$$
Since
\begin{align*}
C&=\bigcup_{n\in\omega}\{(x,y): h(y)(n)=x\}=\\
&=\bigcup_{n\in\omega}\bigcap_{m,k\in\omega}\{x: x(m)=k\}\times\{y: y(b(n,m))=k\},
\end{align*}
we see that $C$ is $F_\sigma$. Now, if $\omega^\omega\es A=\{x_n: n\in\omega\}$, then let us take $y\in\omega^\omega$ such that $h(y)(n)=x_n$ for all $n\in\omega$ and check that:
$$
A=C^y,
$$
which means that $C$ is the desired set.
\vspace{12pt}
\\
\textbf{(ii) Meager subsets of a Polish space $X$:} Let $X$ be uncountable and let $\{U_n: n\in\omega\}$ be an enumeration of basic open sets with $U_0=\emptyset$. We begin by constructing a universal open set $U\se\omega^\omega\times X$ for open dense subsets of $X$. Let us define a function $K: \omega \times \omega \rightarrow \omega$ as follows:
\begin{align*}
    K(0,m) &=0,\\
    K(n,0) &=\min\{k: U_k\se U_n\},\\
    K(n,m+1) &=\min\{k: U_k\se U_n\,\, \land\,\, k>K(n,m)\}.
\end{align*}
$K(n,m)$ is a number of the $(m+1)$th basic open set contained in $U_n$ with respect to our enumeration. Let us set
\[
(x,y)\in U \IFF y\in\bigcup_{n\in\omega}U_{K(n,x(n))}.
\]
Since
\begin{align*}
U&=\bigcup_{n,k\in\omega}\{(x,y): y\in U_{K(n,k)},\, x(n)=k\}=\\
&=\bigcup_{n,k\in\omega}\{x\in \omega^\omega: x(n)=k\}\times U_{K(n,k)},
\end{align*}
we see that $U$ is open. Now let $W=\bigcup_{k\in\omega}U_{n_k}\se X$ be some open dense set. Let $x\in \omega^\omega$ be such that $(\forall k\in\omega) \, (K(n_k,x(n_k))=n_k)$ and $x(m)=0$ for $m\neq n_k, k\in\omega$. Then $W=U_x$.

Now let $h: \baire \rightarrow {\omega^{\omega^\omega}}$ be a homeomorphism. Define a set $G$ by
\[
(x,y)\in G \IFF x\in\bigcap_{n\in\omega}U_{h(x)(n)}.
\]
Clearly, $G$ is a $G_\delta$ set. To check, that it is indeed universal, let $H=\bigcap_{n\in\omega}H_n$, where each $H_n$ is open and dense. Let $x\in \omega^{\omega^\omega}$ be such that $(\forall n\in\omega)(U_{x(n)}=H_n)$. Then $H=G_{h^{-1}(x)}$ and eventually $G^c$ is the desired set.

If we restrict our case to $X=2^\omega$ then using a parametrization of basic meager sets from Lemma \ref{Meagery B-J} we may set
\[
2^\omega\times\omega^\omega\times 2^\omega\es V=\{(x,y,z): x\in 2^\omega, y\in\omega^\omega, z\in F_{x, \widetilde{y}}\},
\]
where $\widetilde{y}$ is a partition of $\omega$ such that $I_0=[0, y(0)+1)$ and $I_{n}=[a_{n-1}, a_{n-1}+y(n)+1)$ for $n>0$, where $a_k$ is the right endpoint of $I_k$. Let $f: \omega^\omega\rightarrow 2^\omega\times\omega^\omega$ be a continuous surjection. Then
\[
U=(f, id)^{-1}[V]
\]
is an $F_\sigma$ universal set for $\mm(2^\omega)$.
\vspace{12pt}
\\
\textbf{(iii) Null subsets of $2^\omega$:} Let $\lambda$ be the Haar measure on $2^\omega$ and for $n\in\omega$ fix an enumeration $\{C_{k}^{n}:k\in \omega\}$ of all clopen subsets of $2^{\omega}$ of measure $<2^{-n}$, with additional requirement that $C^n_0=\0$.
\\
Let $h: \omega^\omega\rightarrow \omega^{\omega\times\omega}$ be a homeomorphism. For each $f\in \omega^{\omega}$ the set
\[
G_{f}=\bigcap_{n}\bigcup_{k>n}C_{h(f)(n, k)}^{n}
\]
is a $G_{\delta}$ null set and 
\[
G=\{(f, y)\in \omega^{\omega}\times 2^{\omega}: y\in G_{f}\}
\]
is a $G_{\delta}$ null subset of the plane. We will show that $G$ is universal for null sets.
\\
Let $X$ be a null set. For each $n\in \omega$ there is a sequence of sets $\{V_{k}^{n}: k\in \omega\}$ from the canonical base such that $X\se \bigcup_{k}V_{k}^{n}$, and $\Sigma_{k}\lambda(V_{k}^{n})<2^{-n-1}$. 
\\
Let us enumerate $\{V_{k}^{n}: n,k\in \omega\}$ into a sequence $\{W_{m}: m \in \omega\}$ and set 
$$
\overline{W}_{m}=\bigcup_{n=a_{m}}^{a_{m+1}-1}W_{n},
$$
where $a_{0}=0$ and for $m>0$
$$
a_{m}=\min\{k>a_{m-1}: \Sigma_{n>k}\lambda(W_{n})<2^{-m}\}+1.
$$
Then $\{\overline{W}_{n}:n\in\omega\}$ is a sequence of clopen sets and
\[
	\lambda(\bigcup_{k>n}\overline{W}_k)\leq\Sigma_{k\geq a_{n+1}}\lambda(W_k)<2^{-n-1}.
\]
Furthermore, if we fix $m\in\omega$ and let
\[
N=\max\{n\in\omega: (\exists k\in\omega, l\leq m)(V_k^n=W_l)\},
\]
then
\[
\bigcup_{k\in\omega} V^{N+1}_k\se \bigcup_{k>m}\overline{W}_k,
\]
therefore $X\se \bigcup_{k>m}\overline{W}_k$ for each $m\in\omega$ and $X\se \bigcap_{n\in\omega}\bigcup_{k>n}\overline{W}_k$. See also that $\bigcup_{k>n}\overline{W}_k=\bigcup_{k\geq a_{n+1}}W_k$. Let us fix $n\in\omega$. Since for each $k\geq a_{n+1}$ we have $\lambda(W_k)<2^{-n}$, it follows that $W_k\in\{C^n_l: l\in\omega\}$. Let $g: \omega\times\omega\to\omega$ be such that $W_k=C^n_{g(n, k)}$ for $k\geq a_{n+1}$ and $g(n, k)=0$ for $k<a_{n+1}$. Then
\[
X\se\bigcap_{n\in\omega}\bigcup_{k>n}C^n_{g(n, k)},
\]
so finally let us set $f=h^{-1}(g)$. $X\se G_f$, thus the proof is completed.
\vspace{12pt}
\\
\textbf{(iv) $\mc{E}\se P(2^\omega)$:} We will begin by constructing a universal set for open sets of full measure. For each $n>0$ let $\{x^n_k: 0\leq k< 2^n\}$ be an enumeration of $\{0, 1\}^n$ such that a sequence $(x^n_0, x^n_1, ..., x^n_{2^n-1})$ is in the lexicographical order. Let us enumerate basic clopen sets of $2^\omega$ in the following way:
\begin{align*}
	B_0&=2^\omega,\\
    B_{n,k}&=\{x\in 2^\omega : x|n=x^n_k\},
\end{align*}
$n\in\omega$, $0\leq k< 2^n$. Let us fix an open set $V$ of full measure. Then
$$
\lim_{n\to\infty}\frac{|\{B_{n,k}\se V: 0\leq k<2^n-1\}|}{2^n}=1.
$$
Let us denote $|\{B_{n,k}\se V: 0\leq k<2^n-1\}|$ by $b_n$. It follows that
$$
(\forall m\in\omega)\, (\exists n>m)\, (\frac{b_n}{2^n}\geq 1-2^{-m}),
$$
otherwise $V$ would not be of full measure. The condition $\frac{b_n}{2^n}\geq 1-2^{-m}$ means, that $V$ contains at least $2^n-2^{n-m}$ clopen sets from $\{B_{n,0}, B_{n,1}, ..., B_{n,2^n-1}\}$. Let $\{A^{m,n}_k: 0\leq k<\binom{2^n}{2^n-2^{n-m}}\}$ be an enumeration of subsets of $\{0, ..., 2^n-1\}$ of size of $2^n-2^{n-m}$. Let
$$
B^m_{n,l}=\bigcup\{B_{n,k}: k\in A^{m,n}_l\}.
$$
Let us define $U\se (\baire)^3\times 2^\omega$ by
$$
(x,y)\in U \IFF y\in\bigcup_{n\in\omega}B^{x_0(n)+n}_{x_1(n), x_2(n) mod \binom{2^{x_1(n)}}{2^{x_1(n)}-2^{x_1(n)-x_0(n)-n}}}.
$$
Then $U$ is a universal set for open sets of full measure. Finally, let
$$
(x,y)\in G \IFF y\in \bigcap_{n\in\omega}U_{h(x)(n)}.
$$
Clearly, $G$ is the desired set.
\vspace{12pt}
\\
\textbf{(v) $\kk_{\sigma}\se P(\omega^{\omega})$:} Here we notice that a relation $\leq^{*}\se (\omega^{\omega})^{2}$ clearly is a universal set which by 
$\leq^{*}=\bigcup_{n}\bigcap_{m>n}\{(x,y):x(m)\leq y(m)\}$ is an $F_{\sigma}$ set.
\vspace{12pt}
\\
\textbf{(vi) Laver ideal:} For $\Phi:\omega^{<\omega}\rightarrow\omega$ we see that the set 
\begin{align*}
   D_{\Phi}&=\{f\in\omega^{\omega}:\existsinfty_{n} f(n)<\Phi(f|_{n})\}=\bigcap_{n}\bigcup_{m>n}\{f\in \omega^{\omega}: f(m)<\Phi(f|_{m})\}
\end{align*}
is clearly a $G_{\delta}$ set and such sets form a base for Laver ideal. Let us consider a set
$$
D=\{(\Phi,y)\in (\omega^{<\omega})^{\omega}\times \omega^{\omega}: \existsinfty_{n} f(n)<\Phi(f|_{n})\}.
$$
Now if we fix a bijection $b:\omega\rightarrow\omega^{<\omega}$ and let $\tilde{b}:\omega^{\omega}\rightarrow(\omega^{<\omega})^{\omega}$ be defined as $\tilde{b}(x)=b\circ x$ for $x\in \omega^{\omega}$ then 
\[
G=\{(x,y)\in\omega^{\omega}\times \omega^{\omega}: \existsinfty_{n} f(n)<(b\circ x)(f|_{n})\}=(\tilde{b},id)^{-1}[D]
\]
is the desired set.
\end{proof}
One could observe that if we take a universal set $U$ for null subsets of $2^\omega$, a universal $F_\sigma$ set $V$, and a homeomorphism $h: \baire \rightarrow \baire\times\baire$, then
$$
\baire\times 2^\omega\es W=\{(x,y): y\in U_{h(x)(0)}\cap V_{h(x)(1)}\}
$$
is a universal set for $\ee$, but the complexity of $W$ may be not optimal.
\\
Also note that a universal set for $\kk_{\sigma}$ subsets of  $\omega^{\omega}$ cannot be $\kk_\sigma$ itself. Let $U$ be such a set and let $h:(\omega^{\omega})^{2}\rightarrow \omega^{\omega}$ be a homeomorphism. Then there exists $f\in \omega^{\omega}$ with $h[U]\leq^{*}f$; this is equivalent to $U\leq^{**} h^{-1}[f]=(f_{1},f_{2})$ where $\leq^{**}$ is coordinatewise product of $\leq^{*}$. Then the set $\{g\in \omega^{\omega}:g\leq^{*} f_{2}+1\}$ is different from any vertical section of $U$. 

\section{Additional results and open questions}

The ideals considered in the previous section have a property that eases the task of finding nice universal sets: they all have Borel bases of bounded complexity. In the light of this fact it is natural to pose a problem:
\begin{prob}
Let $\ii$ be an ideal possessing Borel base of unbounded complexity. Does there exist a $\Sigma^1_1$ or $\Pi^1_1$ universal set for $\ii$?
\end{prob}
In \cite{Cich} the authors showed that the Mokobodzki ideal $\{\0\}\otimes\nn$ has a Borel base of unbounded Borel complexity. Ideals with such a property were studied further in \cite{PBN}, where authors  gave some sufficient conditions for the Fubini product of ideals and its generalizations to have a complex base. In this specific case we may ask the following question.
\begin{prob}
When does there exist a universal set for $\ii\otimes\jj$, provided that there are universal sets for $\ii$ and $\jj$?
\end{prob}
As the last result of this paper we will show that there exist universal sets for some Fubini products of ideals, besides obvious examples like $\nn\otimes\nn$ and $\mm\otimes\mm$.
\begin{lem}\label{Fubini Borel-on-Borel}
If ideals $\ii$ and $\jj$ are $Borel$-on-$Borel$ then $\ii\otimes\jj$ is also $Borel$-on-$Borel$.
\end{lem}
\begin{proof}
Let $X, Y, Z$ be Polish for which $\ii\se P(Y)$ and $J\se P(Z)$ and let $B\se X\times Y\times Z$ be Borel. Let $\widetilde{B}=\{x\in X: B_x\in \ii\otimes\jj\}$. If $x\in\widetilde{B}$ then 
\[
B_x\in\ii\otimes\jj\equiv \{y\in Y: (B_x)_y\notin\jj \}\in\ii.
\]
On the other hand let us consider a set $\widehat{B}=\{(x, y)\in X\times Y: B_{(x,y)}\notin \jj\}$. Since $\jj$ is $Borel$-on-$Borel$, $\widehat{B}$ is Borel. If $x\in X$ is such that $\widehat{B}_x\in\ii$, then
\[\{(x, y)\in X\times Y: B_{(x, y)}\notin\jj\}_x\in\ii\equiv\{y\in Y: B_{(x, y)}\notin\jj\}\in\ii,\]
which means that $\widetilde{B}=\{x\in X: \widehat{B}_x\in\ii\}$. The set $\{x\in X: \widehat{B}_x\in\ii\}$ is Borel, hence we are done.
\end{proof}
In \cite{Bal} the authors pointed out, referring to Fremlin \cite{Fre}, that for each set $B$ from $\nn\otimes\mm$ there exist a $G_\delta$ null set $G\se X$ and an $F_\sigma$ meager set $F\se X^2$ such that $B\se (G\times X)\cup F$. Symmetrically, for each set $B$ from $\mm\otimes\nn$ there is an $F_\sigma$ meager set $F\se X$ and a $G_\delta$ null set $G\se X^2$ such that $B\se (F\times X)\cup G$. Combining these facts with Lemmas \ref{Borel-on-Sigma11} and \ref{Fubini Borel-on-Borel} we obtain the following result.
\begin{cor}
There are $\Sigma^0_3$ universal sets for $\mm\otimes \nn$ and $\nn\otimes \mm$.
\end{cor}

\subsection*{Acknowledgements}
This research was partially supported by the grant S50129/K1102 (0401/0086/16), Faculty of Fundamental Problems of Technology, Wroc{\l}aw University o Science and Technology.

\end{document}